\documentclass[11pt]{article}

\usepackage[a4paper, margin=1in]{geometry}
\usepackage{amsfonts,amsmath,amssymb,amsthm,graphicx,fixmath,latexsym,color}
\usepackage{xcolor}
\usepackage{hyperref,epsfig}
\usepackage{algorithmic}
\usepackage{algorithm}
\usepackage{xspace}

\providecommand{\remove}[1]{}

\newcommand{\HD}{\ensuremath{\mathsf{HD}}\xspace}

\renewcommand{\Re}{\mathbb{R}}

\newcommand{\C}{\mathcal{C}}

\renewcommand{\L}{\mathcal{L}}

\newcommand{\F}{\mathcal{F}}
\newcommand{\G}{\mathcal{G}}

\newcommand{\h}{\mathcal{H}}

\newtheorem{theorem}{Theorem}[section]

\newtheorem{proposition}[theorem]{Proposition}

\newtheorem*{theorem*}{Theorem}
\newtheorem*{lemma*}{Lemma}
\newtheorem*{proposition*}{Proposition}

\begin{document}
\title{A New Lower Bound on Hadwiger-Debrunner Numbers in the Plane}

\author{Chaya Keller\thanks{Department of Mathematics, Technion - Israel Institute of Technology, Haifa, Israel. \texttt{chayak@technion.ac.il}. Research partially supported by Grant 635/16 from the Israel Science Foundation, the Shulamit Aloni Post-Doctoral Fellowship of the Israeli Ministry of Science and Technology, and by the Kreitman Foundation Post-Doctoral Fellowship.}
\and
Shakhar Smorodinsky\thanks{Department of Mathematics, Ben-Gurion University of the NEGEV, Be'er-Sheva, Israel.
\texttt{shakhar@math.bgu.ac.il}. Research partially supported by Grant 635/16 from the Israel Science Foundation.}
}

\date{}
\maketitle

\begin{abstract}
A  family of sets $\F$ is said to satisfy the $(p,q)$ property if among any $p$ sets in $\F$, some $q$ have a non-empty intersection. Hadwiger and Debrunner (1957) conjectured that for any $p \geq q \geq d+1$ there exists $c=c_d(p,q)$, such that any family of compact convex sets in $\mathbb{R}^d$ that satisfies the $(p,q)$ property, can be pierced by at most $c$ points. In a celebrated result from 1992, Alon and Kleitman proved the conjecture. However, obtaining sharp bounds on $c_d(p,q)$, called `the Hadwiger-Debrunner numbers', is still a major open problem in discrete and computational geometry. The best currently known lower bound on the Hadwiger-Debrunner numbers in the plane is $c_2(p,q) = \Omega( \frac{p}{q}\log(\frac{p}{q}))$ while the best known upper bound is $O(p^{(1.5+\delta)(1+\frac{1}{q-2})})$.

In this paper we improve the lower bound significantly by showing that $c_2(p,q) \geq p^{1+\Omega(1/q)}$. Furthermore, the bound is obtained by a family of lines, and is tight for all families that have a bounded VC-dimension. Unlike previous bounds on the Hadwiger-Debrunner numbers which mainly used the weak epsilon-net theorem, our bound stems from a surprising connection of the $(p,q)$ problem to an old problem of Erd\H{o}s on points in general position in the plane. We use a novel construction for the Erd\H{o}s' problem, obtained recently by Balogh and Solymosi using the \emph{hypergraph container method}, to get the lower bound on $c_2(p,3)$. We then generalize the bound to $c_2(p,q)$ for any $q \geq 3$.
\end{abstract}

\newpage

\section{Introduction}
\label{sec:intro}


\paragraph{Helly's theorem, the $(p,q)$ theorem, and Hadwiger-Debrunner numbers.} The classical Helly's theorem asserts that if in some finite family $\F$ of convex sets in $\mathbb{R}^d$, any $d+1$ sets have a non-empty intersection, then the whole family has a non-empty intersection, i.e., it can be {\em pierced} by one point. One of the most challenging extensions of  Helly's theorem was introduced by relaxing the intersection assumption into a weaker assumption called the $(p,q)$ property: Among any $p$ sets in $\F$, some $q$ have a non-empty intersection.

Clearly, not every family that satisfies the $(p,q)$ property has a non-empty intersection; still, one may hope that such a family can be pierced by a `small' number of points. Indeed, Hadwiger and Debrunner~\cite{HD57} conjectured that for all $p \geq q \geq d+1$, any family of convex sets in $\Re^d$ that satisfies the $(p,q)$ property can be pierced by a constant number of points, independent of the size of the family. The minimum such number of points is denoted by $c=c_d(p,q)$. Hadwiger and Debrunner proved their conjecture for the special case when $q>\frac{p}{2}+1$, with $c=p-q+1$; on the other hand, they showed that $p-q+1$ is a lower bound on $c_d(p,q)$ for all pairs $p\geq q$.

After 35 years, the Hadwiger-Debrunner conjecture was proved in a celebrated result of Alon and Kleitman~\cite{AK} also known as the $(p,q)$-Theorem. The upper bound on $c_d(p,q)$ yielded by the proof is $\tilde{O}(p^{d^2+d})$ (for the case $q=d+1$). Alon and Kleitman mentioned that this bound is far from being tight, and since then, the problem of obtaining tight bounds on $c_d(p,q)$ (also called `the Hadwiger-Debrunner numbers' and denoted $\HD_d(p,q)$) is a major open problem in discrete and computational geometry.

Despite extensive research, very little is known about the asymptotics of $\HD_d(p,q)$. Near optimal upper bounds were very recently obtained for very large values of $q$. For example, $\HD_d(p,q) \leq p-q+2$ for all $q>p^{\frac{d-1}{d}+\epsilon}$ \cite{KST18}). Tight bounds were also obtained for specific classes of families (e.g., families of axis-parallel rectangles, see~\cite{Dol72,KS18}), and for specific values of $p,q,d$ (see~\cite{KGT01}). Neither of these results extends to general $(p,q)$.

\paragraph{Weak epsilon-nets and their relation to $\HD_d(p,q)$.} The best currently known lower bounds on $\HD_d(p,q)$ are obtained by lower bounds on the so-called \emph{weak epsilon-nets}. For a finite family of points $\G \subset \mathbb{R}^d$ and for $\epsilon>0$, a weak $\epsilon$-net for $\G$ is a set $S$ of points (not necessarily in $\G$) such that any convex set $T \subset \mathbb{R}^d$ that contains at least $\epsilon |\G|$ points of $\G$, contains also a point of $S$.

Alon et al.~\cite{ABFK} proved that for any $d,\epsilon$ there exists a bound $f_d(\epsilon)$ such that any finite $\G \subset \mathbb{R}^d$ admits a weak $\epsilon$-net of size at most $f_d(\epsilon)$. However, the bound on $f_d(\epsilon)$ was far from being tight, and improving it has been another important open problem. In a very recent breakthrough, Rubin~\cite{Rubin18} showed that for any $\delta>0$, every $\G \subset \mathbb{R}^2$ of size $|\G|>n_0(\delta)$ admits a weak $\epsilon$-net of size at most $\epsilon^{-1.5-\delta}$. This is still far from the best known lower bound $f_d(\epsilon) = \frac{1}{\epsilon} \log^{d-1}(\frac{1}{\epsilon})$ obtained by Bukh, Matou{\v s}ek and Nivasch \cite{BMN11}, which is conjectured to be close to tight.

Weak $\epsilon$-nets are closely related to the $(p,q)$-Theorem. Indeed, for any set of points $\G$, it is easy to see that the family $\F$ of all convex sets that contain at least $\epsilon_0=q/p$ points of $\G$ satisfies the $(p,q)$ property. If the size of the smallest weak $\epsilon_0$-net for $\G$ is $\ell$, then $\F$ is a family of convex sets that satisfies the $(p,q)$ property and cannot be pierced by less than $\ell$ points. Therefore, any lower bound on $f_d(\epsilon)$ translates immediately into a lower bound on $\HD_d(p,q)$. The best known lower bound on $\HD_d(p,q)$ is of this form:
\begin{equation}\label{Eq:Lower-HD}
\HD_d(p,q) = \Omega \left(\frac{p}{q} \log^{d-1} \left(\frac{p}{q} \right)\right),
\end{equation}
following immediately from the aforementioned lower bound of Bukh et al.~\cite{BMN11} on $f_d(\epsilon)$.

While upper bounds on $f_d(\epsilon)$ do not translate directly into upper bounds on $\HD_d(p,q)$, the weak epsilon-net theorem plays a central role in the Alon-Kleitman's proof of the $(p,q)$ theorem, and the best currently known general upper bound for $\HD_d(p,q)$, obtained in~\cite[Proposition~2.6]{KST18}, is formulated in terms of $f_d(\epsilon)$:
\begin{equation}\label{Eq:Upper-HD}
\HD_d(p,q) \leq f_d \left(\Omega(p^{-1-\frac{d-1}{q-d}}) \right).
\end{equation}
In particular, in the plane we have $\HD_2(p,q) =O(p^{(1.5+\delta)(1+\frac{1}{q-2})})$ for any $\delta>0$ and $p>p_0(\delta)$, using Rubin's result~\cite{Rubin18}.

%

\paragraph{Our results.} In this paper we present the first general lower bound on $\HD_d(p,q)$ improving significantly over the best previously known bound $\HD_d(p,q) = \Omega(\frac{p}{q} \log^{d-1}(\frac{p}{q}))$:
\begin{theorem}\label{thm:main}
For any $0<\eta<1/2$ and for any $p,q \geq 3$ such that $q \leq 0.01\eta \cdot (\frac{\log p}{\log \log p})^{1/3}$, there exists a family $\F$ of lines in $\mathbb{R}^2$ which satisfies the $(p,q)$ property and cannot be pierced by less than $p^{1+\frac{1-\eta}{4q-7}}$ points. Consequently, $\HD_d(p,q) \geq p^{1+\frac{1-\eta}{4q-7}}$ for all $d \geq 2$.
\end{theorem}
Interestingly, while our lower bound construction uses a family of lines, which are, in some sense, the `simplest' convex objects, it is tight for a wide class of families -- namely, all families whose so-called \emph{VC-dimension} is bounded.

To explain this statement, a few definitions are needed. For a family of sets $\F$, a set $C$ is said to be \emph{shattered} by $\F$ if the set $\{F \cap C: F \in \F\}$ contains all $2^{|C|}$ subsets of $C$. The VC-dimension of $\F$ is $\sup \{c \in \mathbb{N}: \F \mbox{ shatters some set of cardinality } c\}$. For example, it is easy to see that the VC-dimension of any family of lines is at most $2$.

The notion VC-dimension was introduced by Vapnik and Chervonenkis~\cite{VC71},  and since then found numerous applications (e.g., to computational geometry and to machine learning) and was studied extensively in the last decades (see, e.g.,~\cite{MATOUSEK}). Haussler and Welzl~\cite{HW} proved that any family $\G$ with VC-dimension at most $r$ admits a weak $\epsilon$-net (and actually, the significantly stronger notion of `$\epsilon$-net', see~\cite{MATOUSEK}) of size $O(\frac{r}{\epsilon} \log(\frac{r}{\epsilon}))$.

Substituting the assertion of the Haussler-Welzl theorem into~\eqref{Eq:Upper-HD}, we obtain the upper bound $\HD_2(p,q) \leq O\left(p^{1+\frac{1}{q-2}} \log p \right)$ for any finite family $\F$ of convex sets in the plane with a bounded VC-dimension. Therefore, Theorem~\ref{thm:main} shows that within the class of families with a bounded VC-dimension we have $c_d(p,q) = p^{1+\Theta(1/q)}$.

\paragraph{Connection to a problem of Erd\H{o}s on points in general position in the plane.}

While the best previously known bounds on the Hadwiger-Debrunner numbers were obtained via improved bounds for the weak epsilon-net theorem, our bound stems from a surprising connection between the $(p,q)$ problem and an old problem of Erd\H{o}s regarding points in general position in the plane.

In~\cite{ERDOS1986}, Erd\H{o}s raised the following problem: What is the maximal possible $\ell=\ell(n)$ such that any set $S$ of $n$ points with no $4$ of them collinear, contains a subset of size $\ell$ in general position (that is, with no three collinear points)?

Until recently, the best known upper bound for Erd{\H{o}}s problem was $\ell(n)=o(n)$, proved by F\"{u}redi~\cite{Furedi:1991} using the Density Hales-Jewett theorem of Katznelson and Furstenberg~\cite{FURSTENBERG1989,Furstenberg1991}. In a major breakthrough, Balogh and Solymosi~\cite{BS17} proved that $\ell(n) \leq n^{5/6+\delta}$, for any $\delta>0$ and any $n>n_0(\delta)$.

The result of Balogh and Solymosi is based on a random subhypergraph of the hypergraph $\h(n,3,3)$ whose vertices are the points in the three-dimensional grid $[n]^3$, and whose hyperedges are triples of collinear points. Essentially, first a subset of the vertices of $\h(n,3,3)$ of size $\approx n^2$ is chosen randomly, and then a small part of the vertices is removed in such a way that the resulting subset contains no $4$ collinear points, while any subset of it of size $n^{\frac{5}{3}+\delta}$ contains a collinear triple. Finally, the chosen set is projected into the plane in such a way that collinearity is preserved and no new collinear triples are created. The heart of the result is an upper bound on the number of independent subsets of $\h(n,3,3)$ of size $n^{\frac{5}{3}+\delta}$ (namely, sets of $n^{\frac{5}{3}+\delta}$ vertices that do not contain a collinear triple). This upper bound is obtained using the recently proposed \emph{hypergraph container method}~\cite{BMS15,ST15}, in the first application of the method to combinatorial geometry.

\medskip

The main observation underlying our results is that an upper bound for Erd\H{o}s problem is directly translated into a lower bound on $\HD_2(p,q)$. Indeed, let $S$ be a set of $n$ points in the plane with no collinear $4$-tuple, such that any subset of $S$ of size at least $\ell(n)$ contains a collinear $3$-tuple. By point-line duality in the plane, we can transform $S$ into a family $\F$ of $n$ lines, such that no $4$ lines share a common point, while each subset of $\F$ of size $\ell(n)$ contains three lines with a common point. The latter condition means exactly that $\F$ satisfies the $(\ell(n),3)$ property. On the other hand, the former condition implies that $\F$ cannot be pierced by less than $n/3$ points. Hence, $\F$ is a family of convex sets in the plane that satisfies the $(\ell(n),3)$ property but cannot be pierced by less than $n/3$ points, and thus, $\HD_2(\ell(n),3) \geq n/3$.

Combining this observation with the result of Balogh and Solymosi, we immediately obtain the lower bound
\[
\HD_2(p,3) \geq p^{\frac{6}{5}-\delta},
\]
for all $\delta>0$ and $p>p_0(\delta)$, which is the assertion of Theorem~\ref{thm:main} in the case $q=3$. The result for a general $q \geq 3$ is much more involved, and requires generalizing the construction of Balogh and Solymosi and their argument to random subsets of the hypergraph $\h(n,2q-2,q)$ whose vertices are the points in the $(2q-2)$-dimensional grid, and whose hyperedges are collinear $q$-tuples. Interestingly, the choice of dimension is crucial for obtaining Theorem~\ref{thm:main}; applying the same technique with the $q$-dimensional grid leads to a significantly weaker result.

\paragraph{Application to a hypergraph coloring problem.} As an additional demonstration of our proof method, we apply it to a natural problem on coloring geometric hypergraphs.

The following problem was implicitly stated by Payne and Wood~\cite[Section~4]{PW13} (though, using a different terminology). For a set $P$ of $m$ points in the plane, let $H_{q}(P)$ be the (non-uniform) hypergraph whose vertex set is $P$ and whose hyperedges are all sets $\{S \subset P: S = P \cap \ell \mbox{ for some line $\ell$ }, |S| \geq q\}$.
Find the maximal possible chromatic number of $H_q(P)$ as a function of $m$, i.e., determine
\[
g_q(m) = \max_{|P|=m} \chi(H_q(P)).
\]
This is a natural question that was studied for various other geometric shapes. For example, the same problem for $q=3$, with lines replaced by axis-parallel rectangles, was studied by Ackerman and Pinchasi~\cite{AP13}.

A relatively simple probabilistic argument using the Lov\'{a}sz Local Lemma shows that $g_q(m)=O(m^{1/(q-1)})$ for all $m,q$. Using the proof method of Theorem~\ref{thm:main}, we obtain the following lower bound on $g_q(m)$.
\begin{proposition}\label{Prop:Coloring}
For any $\eta>0$, $q \geq 3$ and any $m \in \mathbb{N}$ such that $q \leq 0.005\eta (\frac{\log m}{\log \log m})^{1/4}$, we have
\[
g_q(m) \geq m^{\frac{1-\eta}{q^2-q-\eta}}.
\]
\end{proposition}

\paragraph{Organization of the paper.} In Section~\ref{sec:prelim} we introduce some notations and prove a \emph{super-saturation} property of the hypergraph $\h(n,k,r)$ that will be used in the proof of Theorem~\ref{thm:main}. The proofs of Theorem~\ref{thm:main} and Proposition~\ref{Prop:Coloring} are presented in Section~\ref{sec:main}.

\section{Preliminaries}
\label{sec:prelim}


\subsection{Notations}

\paragraph{The hypergraph $\h=\h(n,k,r)$.} For $n,k,r \geq 3$, we define $\h=\h(n,k,r)$ to be the $r$-uniform hypergraph whose vertices are the points of the grid $[n]^k = \{1,2,\ldots,n\}^k$, such that $r$ points form a hyperedge if and only if they are collinear. Clearly, $|V(\h)|=n^k$.

\paragraph{Induced subhypergraph.} For a hypergraph $H=(V,E)$ and for a subset $V'$ of vertices, we denote by $H[V']$ the subhypergraph whose vertex set is $V'$ and whose hyperedges are $\{e \in E(H): e \subset V'\}$. We refer to $H[V']$ as the {\em induced subhypergraph} of $H$ on $V'$. Note that this definition is somewhat non-standard; the more common definition is taking the hyperedges to be $\{e \cap V': e \in E(H)\}$. However, throughout this paper we mostly deal with uniform hypergraphs for which the more common definition does not make much sense.

\paragraph{Degree and co-degree.} For a hypergraph $H=(V,E)$ and a vertex $v \in V$, the {\em degree} of $v$ is $\deg(v)=|\{e \in E: v \in e\}|$. For a set of vertices $S \subset V$, the {\em co-degree} of $S$ is $\deg(S) = |\{e \in E: S \subset e\}|$.

\paragraph{Independent set.} A set $V'$ of vertices in a hypergraph $H=(V,E)$ is called {\em independent} if it does not contain any hyperedge.

\paragraph{Coloring of a hypergraph.} A {\em proper coloring} of a hypergraph $H=(V,E)$ with $c$ colors is a function $f:V \rightarrow [c]$ such that no hyperedge is monochromatic, or equivalently, if each color class $f^{-1}(\{i\})$ is an independent set. The {\em chromatic number} of $H$, denoted by $\chi(H)$, is the minimum number of colors sufficient for a proper coloring of $H$.

\paragraph{Intersecting $q$-tuples.} For sake of brevity, we call a family that consists of $q$ sets whose intersection is non-empty an {\em intersecting $q$-tuple}. We note that throughout the paper `tuples' are \emph{unordered}.

\paragraph{Generalized Binomial coefficients.} We use the generalized Binomial coefficient defined as follows: for any $x \in \mathbb{R}$ put
$${{x}\choose{k}} = \frac{x(x-1)\cdot \ldots \cdot (x-k+1)}{k!}$$.

\paragraph{Logarithms.} All logarithms used in the paper are with respect to the natural basis, unless explicitly mentioned otherwise.

\subsection{A super-saturation property of the hypergraph $\h(n,k,r)$}

In this subsection we show that any subhypergraph of $\h(n,k,r)$ with a `not-too-small' number of vertices contains a non-negligible portion of the hyperedges of $\h(n,k,r)$. This \emph{super-saturation property} of $\h(n,k,r)$ will be a key ingredient in the application of the hypergraph container method in Proposition~\ref{Prop:Container} below.

The super-saturation property follows from two propositions: the first is an upper bound on $|E(\h(n,k,r)|$, while the second is a lower bound on $|E(\h(n,k,r)[V'])|$, for any $V'$ that is `not-so-small'. Our presentation in this section follows the argument of~\cite[Section~4]{BS17}, and we try to follow the same notations as in~\cite{BS17} as much as possible.


The following easy bound was proved in~\cite[Claim~4.1]{BS17}.
\begin{proposition}\label{Prop:Size_of_graph}
For any $n \geq k,r$, the number of hyperedges in the hypergraph $\h(n,k,r)$ satisfies:
\[
|E(\h(n,k,r))| \leq \left\lbrace
  \begin{array}{c l}
    \frac{k \cdot 2^{r+k}}{r!} \cdot n^{2k}, & r \leq k\\
    \frac{k \cdot 2^{r+k}}{r!} \cdot n^{2k} \cdot \log n, & r = k+1\\
    \frac{k \cdot 2^{r+k+1}}{r!} \cdot n^{r+k-1}, & r>k+1.
  \end{array}
\right.
\]
\end{proposition}
We note that Proposition~\ref{Prop:Size_of_graph} was stated in~\cite{BS17} only for $r \leq 2k$. However, exactly the same proof applies for $r>2k$ as well.

\medskip

The following proposition is a generalization of~\cite[Lemma~4.2]{BS17}, where the same assertion is proved for $r=k=3$. Since we are interested in the exact dependence of the bound on $k,r$, we present the proof.\footnote{We note that the proof in~\cite[Lemma~4.2]{BS17} contains several inaccuracies that are fixed here. In particular, the sets $U,V$ cannot be defined as in~\cite{BS17} and so we modify their definition.}
\begin{proposition}\label{Prop:Size-of-subgraph-lower}
For all $s \in [0,0.9]$, $k,r \geq 3$ and $n \geq \max(e^{100k},r^{100})$, the following holds. Let $S$ be a set of vertices of $\h(n,k,r)$ of size $n^{k-s}$. Then
\[
|E(\h(n,k,r)[S])| \geq \frac{n^{2k-(k+1)s}}{r^{k+1} \cdot (1000 \cdot 9^k)^{k+1} \cdot \log n}.
\]
\end{proposition}

\begin{proof}
The proof is constructive, showing an explicit set of lines that contain many collinear $r$-tuples from any subset of $[n]^k$ of size $n^{k-s}$. Let $S \subset V(\h(n,k,r))$ be such that $|S|=n^{k-s}$. Set $t=c_0 n^s$, where $c_0=c_0(k,r)$ is a parameter that will be determined below in such a way that the condition $t \leq n^{0.99}$ will hold. Define
\[
U = \{(a_1,a_2,\ldots,a_k) \in \mathbb{Z}^k: 1 \leq a_1 \leq \frac{2n}{t}, -n \leq a_2,\ldots,a_k \leq n \}
\]
and
\[
V =\{(a_1,a_2,\ldots,a_k) \in \mathbb{Z}^k: \frac{n}{t} \leq a_1 \leq \frac{2n}{t}, 0 \leq a_2,\ldots,a_k < a_1, a_1 \mbox{ is a prime number} \}.
\]
It is clear that
\begin{equation}\label{Eq:Size-of-U}
|U| \leq \frac{2n}{t} \cdot (2n+1)^{k-1} \leq 3^k \cdot \frac{n^k}{t}.
\end{equation}
Furthermore, as for all $m > 1$, the number of primes between $1$ and $m$ is at most $1.25506m/\log m$ (see~\cite[Corollary~1]{rosser1962}), we have on one hand:
\begin{equation}\label{Eq:Size-of-V-upper}
|V| \leq \frac{1.26 \cdot \frac{2n}{t}}{\log(2n/t)} \cdot \left(\frac{2n}{t}\right)^{k-1} \leq 3^k \cdot \frac{n^k}{t^k \log (2n/t)} \leq 100 \cdot 3^k \cdot \frac{n^k}{t^k \log n},
\end{equation}
where the last inequality holds since $t \leq n^{0.99}$. On the other hand, as for all $m \geq 17$, the number of primes between $1$ and $m$ is at least $m/\log m$ (see~\cite[Corollary~1]{rosser1962}), and by assumption, $\frac{2n}{t} \geq 2n^{1/100} \geq 2e^k > 17$, we have
\begin{equation}\label{Eq:Size-of-V-lower}
|V| \geq \left(\frac{\frac{2n}{t}}{\log(2n/t)} - \frac{1.26 \cdot \frac{n}{t}}{\log(n/t)} \right) \left(\frac{n}{t}\right)^{k-1} \geq \frac{0.1n^k}{t^k \log n}.
\end{equation}
Let $\L=\L(t)$ be the family of all lines containing points in $U$ and having directions in $V$. That is, for each $u \in U, v \in V$ we let $L(u,v)=\{u+\alpha v:\alpha \in \mathbb{R}\}$, so $\L = \{L(u,v): u \in U, v \in V\}$. We show that the number of collinear $r$-tuples from $S$ contained in lines $L(u,v)$ is larger than $\frac{n^{2k-(k+1)s}}{r^{k+1} \cdot (1000 \cdot 9^k)^{k+1} \cdot \log n}$, thus proving the assertion. We achieve this in four steps:
\begin{enumerate}
\item We obtain an upper bound on the size of $\L$.
\item We obtain a lower bound on the number of lines in $\L$ that pass through any specific point in $[n]^k$.
\item We obtain a lower bound on the number of incidences between points of $S$ and lines of $\L$.
\item Using the bounds on the number of incidences and on the size of $\L$, along with a convexity argument, we deduce a lower bound on the number of collinear $r$-tuples in $S$ included in lines of $\L$.
\end{enumerate}

\paragraph{Step 1: Bounding the size of $\L$.} This step is immediate. By Equations~\eqref{Eq:Size-of-U} and~\eqref{Eq:Size-of-V-upper}, we have
\begin{equation}\label{Eq:Size-of-L-upper}
|\L| \leq |U| |V| \leq 3^k \cdot \frac{n^k}{t} \cdot 100 \cdot 3^k \cdot \frac{n^k}{t^k \log n} = 100 \cdot 9^k \cdot \frac{n^{2k}}{t^{k+1}\log n}.
\end{equation}

\paragraph{Step 2: Bounding the number of lines in $\L$ that pass through any point in $[n]^k$.} For each point $x \in [n]^k$ and for each $v \in V$, there exists $u \in U$ such that $x \in L(u,v)$. Indeed, the sequence $\{x+jv: j \in \mathbb{Z}\}$ contains at least one point $u$ with $1 \leq u_1 \leq 2n/t$, and since $0 \leq v_2,\ldots,v_k<v_1$, $u$ must satisfy $-n \leq u_2,\ldots,u_k \leq n$. Thus, $u \in U$ and $x \in L(u,v)$.

Furthermore, we claim that if $v,v'$ are distinct elements of $V$ then for any $u,u' \in U$ we have $L(u,v) \neq L(u',v')$. Indeed, note that $L(u,v)= L(u',v')$ may hold only if $v' = \beta v$ for some $1 \neq \beta \in \mathbb{R}$, and in particular, $v'_1 v_2 = v_1 v'_2$. Assume to the contrary that equality holds for some $v,v'$ with $v'_1>v_1$. The prime number $v'_1$ divides the left hand side of the equation but not the right hand side, as $\max(v_1,v'_2)<v'_1$ and as $v'_1$ is a prime -- a contradiction.

It follows that each $x \in [n]^k$  belongs to at least $|V|$ distinct lines of the form $L(u,v)$. Using Equation~\eqref{Eq:Size-of-V-lower} we get:
\begin{equation}\label{Eq:Lines-through-point-lower}
\forall x \in [n]^k, |\{L \in \L: x \in L\}| \geq \frac{0.1n^k}{t^k \log n}.
\end{equation}

\paragraph{Step 3: Bounding the number of incidences between points of $S$ and lines of $\L$.} By Equation~\eqref{Eq:Lines-through-point-lower}, the number of incidences between points in $S$ and lines in $\L$ is at least
\[
n^{k-s} \cdot \frac{0.1n^k}{t^k \log n} = \frac{0.1n^{2k-s}}{t^k \log n}.
\]

\paragraph{Step 4: Bounding the number of collinear $r$-tuples.} We now count the collinear $r$-tuples included in $S$ by going over the lines in $\L$ and counting the number of collinear $r$-tuples on each line. Since the function $g$ defined by $g(x)= {{x}\choose{r}}$ if $x \geq r$ and $g(x)=0$ otherwise is convex, once the number of incidences is fixed, the number of collinear $r$-tuples is minimized when $\L$ is as large as possible and the numbers of points on all lines are equal. Substituting the upper bound on $|\L|$ obtained in Equation~\eqref{Eq:Size-of-L-upper}, we get that in this case, the average number of points of $S$ on a line in $\L$ is
\[
\left(\frac{0.1n^{2k-s}}{t^k \log n}\right) \Big{/} \left(100 \cdot 9^k \cdot \frac{n^{2k}}{t^{k+1}\log n}\right) = \frac{t}{1000 \cdot 9^k n^{s}} = \frac{c_0}{1000 \cdot 9^k},
\]
where the last equality follows from the definition of $t$.
Hence, the number of collinear $r$-tuples included in $S$ is lower bounded by
\[
{{c_0 / (1000 \cdot 9^k)}\choose{r}} \cdot 100 \cdot 9^k \cdot \frac{n^{2k}}{t^{k+1}\log n},
\]
assuming $c_0 / 1000 \cdot 9^k \geq r$. (Otherwise, we get a trivial lower bound.) In order to satisfy the assumption we take $c_0 = r \cdot 1000 \cdot 9^k$. Note that provided $s_0 \leq 0.9$ and $n \geq \max(e^{100k},r^{100})$, the condition $t = c_0 n^s \leq n^{0.99}$ is satisfied. With this choice of $t$, we obtain a lower bound of
\[
{{r}\choose{r}} \cdot 100 \cdot 9^k \cdot \frac{n^{2k}}{t^{k+1}\log n} \geq 100 \cdot 9^k \frac{n^{2k}}{r^{k+1} \cdot 1000^{k+1} \cdot 9^{k(k+1)} n^{(k+1)s}\log n} \geq \frac{n^{2k-(k+1)s}}{r^{k+1} \cdot (1000 \cdot 9^k)^{k+1} \cdot \log n}
\]
on the number of collinear $r$-tuples, completing the proof.
\end{proof}

\section{Proof of the main theorem}
\label{sec:main}

In this section we prove Theorem~\ref{thm:main}. Let us recall its statement.

\medskip \noindent \textbf{Theorem~\ref{thm:main}.}
For any $0<\eta<1/2$ and for any $p,q \geq 3$ such that $q \leq 0.01\eta \cdot (\frac{\log p}{\log \log p})^{1/3}$, there exists a family $\F$ of lines in $\mathbb{R}^2$ which satisfies the $(p,q)$ property and cannot be pierced by less than $p^{1+\frac{1-\eta}{4q-7}}$ points. Consequently, $\HD_2(p,q) \geq p^{1+\frac{1-\eta}{4q-7}}$.

\medskip \noindent The proof of the theorem consists of three stages:
\begin{enumerate}
\item \emph{Reduction stage.} We show that it is sufficient to prove that for some $n,k,u$, there exists a subset $S$ of $[n]^k$ of size at least $(u-1) \cdot p^{1+\frac{1-\eta}{4q-7}}$ that does not contain collinear $u$-tuples and also does not contain independent sets of size at least $p$ of the hypergraph $\h(n,k,q)$.

\item \emph{Upper bound on the number of independent $m$-subsets of $\h(n,k,r)$.} We obtain an upper bound on the number of independent subsets of size $m$ of the hypergraph $\h(n,k,r)$, as function of $n,k,r,m,$ and auxiliary parameters $s_0,f$ to be defined below. The idea behind this stage is apparent: if the number of independent subsets of size $p$ of $\h(n,k,q)$ is `small', then it is easier for a randomly chosen subset of the vertices of $\h(n,k,q)$ to be free of independent sets of size $p$. This stage uses the \emph{hypergraph container method}.

\item \emph{Probabilistic construction.} We construct the required set $S$ using the probabilistic method. Specifically, we consider an $\alpha$-random subset $\tilde{S}$ of $[n]^k$ for some $n,k,\alpha$. We show that for an appropriate choice of all involved parameters, with a positive probability $\tilde{S}$ does not contain independent sets of $\h(n,k,q)$ of size $p$ and contains only a small amount of collinear $u$-tuples, so that we can remove them and obtain a set $S$ of size at least $(u-1) \cdot p^{1+\frac{1-\eta}{4q-7}}$ with no collinear $u$-tuples and no independent subsets of $\h(n,k,q)$ of size $p$.
\end{enumerate}
The proof method we use follows (and generalizes) the argument of~\cite[Section~5]{BS17}, and we try to use the same notations as in~\cite{BS17} as much as possible.

The three stages of the proof are presented in the following three subsections. We conclude this section with an application of our proof method to a natural geometric hypergraph coloring problem in Section~\ref{sec:sub:coloring}.

\subsection{Reduction to subsets of $[n]^k$}

The easy reduction is obtained in the following proposition.
\begin{proposition}\label{Prop:Reduction}
Let $S \subset [n]^k$ be a set of points such that:
\begin{enumerate}
\item $S$ does not contain an independent set of size $p$ of the hypergraph $\h(n,k,q)$;

\item $S$ does not contain $u$ collinear points;

\item $|S| \geq (u-1) \cdot p^{1+\frac{1-\eta}{4q-7}}$.
\end{enumerate}
Then $S$ can be transformed into a family $\F$ of lines in $\mathbb{R}^2$ that satisfies the assertion of Theorem~\ref{thm:main}.
\end{proposition}

\begin{proof}
Let $S$ be a set of vertices that satisfies the hypothesis. The set $S$ can be projected into a set $S'$ of $|S|$ points in the plane in such a way that collinear point tuples stay on a line, and no new collinear point tuples are created. Applying an incidence-preserving point-line duality in $\mathbb{R}^2$, the set $S'$ can be transformed into a family $\F$ of $|S|$ lines in the plane in such a way that a set of lines in $\F$ has a common point if and only if the corresponding points in $S'$ are collinear.

By Condition~(1), any set of $p$ points in $S$ contains a collinear $q$-tuple. Hence, any set of $p$ lines in $\F$ contains a $q$-tuple of lines whose intersection is non-empty. That is, $\F$ satisfies the $(p,q)$ property.

By Condition~(2), $S$ does not contain $u$ collinear points. Thus, $\F$ does not contain $u$ lines whose intersection is non-empty. Consequently, $\F$ cannot be pierced by less than
\[
\frac{|\F|}{u-1} \geq \frac{(u-1) \cdot p^{1+\frac{1-\eta}{4q-7}}}{u-1} = p^{1+\frac{1-\eta}{4q-7}}
\]
points, where the inequality uses Condition~(3) and the equality $|\F|=|S|$. Therefore, $\F$ satisfies the assertion of Theorem~\ref{thm:main}.
\end{proof}

\subsection{On containers and independent subsets of $\h(n,k,r)$}
\label{sec:sub:container}

In this subsection we obtain an upper bound on the number of independent subsets of $\h(n,k,r)$ of a given size $m$. Following~\cite{BS17}, we use the hypergraph container method~\cite{BMS15,ST15} which has proved to be extremely powerful in obtaining such upper bounds. We start with a very brief description of the method and then we apply it in our case.

\subsubsection{The hypergraph container method}

The hypergraph container method was introduced independently by Saxton and Thomason~\cite{ST15} and by Balogh, Morris, and Samotij~\cite{BMS15}. Intuitively, for a hypergraph $H=(V,E)$ whose co-degrees are `distributed evenly', the method allows finding a relatively small family $\C$ of `not-too-large' subsets of $V$ called `containers', such that each independent set in $V$ is included in some container $C \in \C$. This, in turn, allows to bound the number of independent sets of any fixed size, as shown below.

In the few years since the method was introduced, it was applied to numerous problems in extremal graph theory, Ramsey theory, and additive combinatorics (see the survey~\cite{BMS18}). The application of the method to discrete geometry was pioneered by Balogh and Solymosi~\cite{BS17}, whose route we follow here.

The version of the method we use (i.e., Theorem~\ref{thm:container} below) yields an effective bound on $|\C|$ but does not provide a bound on the size of each container. Instead, it asserts that each container contains only a few hyperedges. This version can be used along with a super-saturation lemma which asserts that if some induced subhypergraph of $H$ has only a few hyperedges then it cannot have too many vertices. Given such a super-saturation result (which we obtained in Proposition~\ref{Prop:Size-of-subgraph-lower} above), one can apply Theorem~\ref{thm:container} sequentially a bounded number of times such that eventually, all containers become sufficiently small.

\medskip

In order to present the method, we need a few more notations. For an $r$-uniform hypergraph $H$ with an average degree $d$, and for every $j \in [r]$, let $\Delta_j$ be the maximum co-degree of a set of $j$ vertices, i.e., $\Delta_j = \max_{|S|=j} d(S)$. For $0<\tau<1$, denote
\[
\Delta(H,\tau) = 2^{{{r}\choose{2}}-1} \sum_{j=2}^r \frac{\Delta_j}{d \tau^{j-1} 2^{{{j-1}\choose{2}}}}.
\]
We use the following version of the hypergraph container theorem~\cite[Corollary~3.6]{ST15}.
\begin{theorem}[Saxton and Thomason]\label{thm:container}
Let $H$ be an $r$-uniform hypergraph on $N$ vertices. Let $0 < \epsilon,\tau < 1/2$. Suppose that we have
\[
\tau < \frac{1}{200r \cdot (r!)^2} \qquad \mbox{ and } \qquad \Delta(H,\tau) \leq \frac{\epsilon}{12r!}.
\]
Then there exists $c=c(r) \leq 2000r \cdot (r!)^3$ and a collection $\mathcal{C}$ of vertex sets such that:
\begin{enumerate}
\item Every independent set in $H$ is contained in some $A \in \mathcal{C}$;

\item For every $A \in \mathcal{C}$, we have $|E(H[A])| \leq \epsilon |E(H)|$; and

\item We have $\log |\mathcal{C}| \leq c N \tau \log(1/\epsilon) \cdot \log(1/\tau)$.
\end{enumerate}
\end{theorem}

\subsubsection{An upper bound on the number of independent $m$-subsets of $\h(n,k,r)$}

\begin{proposition}\label{Prop:Container}
Let $0<f<s_0<0.9$, and let $n,k,r$ be natural numbers such that:
\begin{enumerate}
\item $s_0 \leq \frac{k-r+1}{k}$, and in particular, $k \geq r$;

\item $f \geq \frac{10^4 \log \log n}{\log n}$;

\item $k \leq 0.001f \cdot \frac{\log n}{\log \log n}$.
\end{enumerate}
Then for any $m \in [n]$, the number of independent sets of size $m$ in the hypergraph $\h(n,k,r)$ is at most
\[
\exp \left(n^{k-s_0-\frac{k-ks_0}{r-1}+0.3f} \right) \cdot {{n^{k-s_0+0.1f}}\choose{m}}.
\]
\end{proposition}

\noindent \textbf{Remark.} Before we present the proof, two remarks are due regarding the auxiliary parameters $s_0,f$ and Conditions~(1),(2),(3).

\medskip \noindent \emph{The parameter $f$ and Conditions~(2),(3).} The parameter $f$ is a `small error term', intended also for absorbing all low-order terms for sake of clarity. Specifically, Conditions~(2),(3) allow us to neglect all terms of the form $k!$, $r!$, $\log n$ etc. we encounter during the proof; we absorb each of them into the term $n^{f}$ (or more precisely, into the term $n^{cf}$ for a small constant $c$) immediately after its first appearance. Similarly, we use Condition~(3) to absorb terms of the form $2^{k^2}$ into the error term $n^{kf}$.

\medskip \noindent \emph{The parameter $s_0$ and Condition~(1).} This parameter helps us to determine an upper bound on the size of containers we want to achieve. Specifically, we continue applying the hypergraph container theorem sequentially until all containers are of size at most $n^{k-s_0+0.1f}$. The assumption $s_0 \leq \frac{k-r+1}{k}$ (i.e., Condition~(1)) allows us to simplify the analysis as we show below, and will be sufficient for our purposes (as shown in Appendix~\ref{app:Large-s_0}). The analysis can be performed also for larger values of $s_0$ but the result becomes more cumbersome.

\begin{proof}[Proof of Proposition~\ref{Prop:Container}]
We obtain the assertion by a sequence of applications of the hypergraph container theorem (i.e., Theorem~\ref{thm:container} above).

\paragraph{Sequential application of Theorem~\ref{thm:container}.} We start with the hypergraph $\h=\h(n,k,r)$ and introduce the notations $C^0=V(\h)=[n]^k$, $\C^0=\{C^0\}$ and $\h^0=\h[C^0]$. At Step~1, we apply Theorem~\ref{thm:container} to the hypergraph $\h^0$ and obtain a family $\C^1$ of containers. At Step~2, we consider each container $C^1_j \in \C^1$ and if $|C^1_j|>n^{k-s_0+0.1f}$ (i.e., if $C^1_j$ is not sufficiently small yet), we apply Theorem~\ref{thm:container} to the hypergraph $\h^1_j=\h[C^1_j]$. We denote by $\bar{\C}^2$ the family of all containers obtained in Step~2 (from all elements of $\C^1$) and set $\C^2 = \bar{\C}^2 \cup \{C^1_j \in \C^1: |C^1_j| \leq n^{k-s_0+0.1f}\}$ (i.e., adding to $\bar{\C}^2$ all elements of $\C^1$ which were `sufficiently small' so that Theorem~\ref{thm:container} wasn't applied to them). At Step~3, we repeat the procedure with $\C^2$ instead of $\C^1$. We continue in this fashion until for some $l$, all containers in $\C^{l}$ are of size $\leq n^{k-s_0+0.1f}$, and denote that final family of containers by $\C$.

\paragraph{Bounding the number of steps via the choice of $\epsilon$.} In all applications of Theorem~\ref{thm:container}, we take $\epsilon = n^{-0.05fk}$. As a result, the number of hyperedges of the hypergraph to which Theorem~\ref{thm:container} is applied shrinks by a factor of $n^{0.05fk}$ every time. On the other hand, the number of hyperedges in any hypergraph $\h^i_j$ to which Theorem~\ref{thm:container} is applied during our process can be bounded from below using Proposition~\ref{Prop:Size-of-subgraph-lower}, as otherwise $V(\h^i_j)<n^{k-s_0+0.1f}$ and $\h^i_j$ is `already sufficiently small'.
Hence, we can use the ratio between the number of hyperedges in $C^0$ and the number of hyperedges for which we stop applying Theorem~\ref{thm:container} to bound the number of steps in our process.

Specifically, by Proposition~\ref{Prop:Size_of_graph} we have
\[
|E(C^0)| = |E(\h(n,k,r))| \leq \frac{k \cdot 2^{r+k}}{r!} \cdot n^{2k} \leq n^{2k+0.01f} \leq n^{2k+0.01kf},
\]
where the penultimate inequality uses Condition~(3). By Proposition~\ref{Prop:Size-of-subgraph-lower}, for any $i,j$ for which Theorem~\ref{thm:container} is applied to $\h^i_j$ we have
\[
|E(\h^i_j)| \geq \frac{n^{2k-(k+1)(s_0-0.1f)}}{r^{k+1} \cdot (1000 \cdot 9^k)^{k+1} \cdot \log n} \geq n^{2k-(k+1)(s_0-0.1f)-0.01fk}
\]
(using Conditions~(2),(3)). Therefore, the process ends after at most
\[
\frac{(2k+0.01fk)-(2k-(k+1)(s_0-0.1f)-0.01fk)}{0.05fk} \leq \frac{40}{f}
\]
steps, which guarantees that the size of $\C$ will not be `too large'.

\paragraph{Reduction to a single application of the hypergraph container theorem.} It clearly follows from Theorem~\ref{thm:container} that for each $i$, any independent set in $\h$ is contained in some element of $\C^i$. Hence, by a union bound, for each $m$, the number of independent sets of size $m$ in $\h$ is at most $|\C| \cdot {{n^{k-s_0+0.1f}}\choose{m}}$. Consequently, in order to prove the assertion it is sufficient to show that $|\C| \leq \exp \left(n^{k-s_0-\frac{k-ks_0}{r-1}+0.3f} \right)$. As there are at most $40/f$ steps, and as $\frac{40}{f} \leq n^{0.1f}$ by Condition~(2), it is sufficient to show that in each single application of Theorem~\ref{thm:container}, the size of each obtained family of containers is at most $\exp \left(n^{k-s_0-\frac{k-ks_0}{r-1}+0.2f} \right)$.

\paragraph{Analysis of a single application of Theorem~\ref{thm:container}.} Consider a single application of Theorem~\ref{thm:container} at Step~$(i+1)$, i.e., an application of the theorem to some $\h^i_j=\h[C^i_j]$. Let $s=s(i,j)$ be such that $|C^i_j|=n^{k-s}$. Note that $s \leq s_0-0.1f$, as otherwise $C^i_j$ is already `sufficiently small'.

\medskip \noindent \emph{Bounding $\Delta(\h^i_j,\tau)$ as function of $\tau$.} By Proposition~\ref{Prop:Size-of-subgraph-lower}, we have
\[
|E(\h^i_j)| \geq \frac{n^{2k-(k+1)s}}{r^{k+1} \cdot (1000 \cdot 9^k)^{k+1} \cdot \log n},
\]
and thus the average degree $d$ of $\h^i_j$ satisfies
\[
d \geq \frac{r}{n^{k-s}} \cdot \frac{n^{2k-(k+1)s}}{r^{k+1} \cdot (1000 \cdot 9^k)^{k+1} \cdot \log n} = \frac{n^{k-ks}}{r^k \cdot (1000 \cdot 9^k)^{k+1} \cdot \log n} \geq n^{k-ks-0.01kf},
\]
where the last inequality follows from Conditions~(2),(3) above.

For each set $S$ of vertices of $\h^i_j$, the co-degree of $S$ is at most ${{n-|S|}\choose{r-|S|}} \leq n^{r-|S|}$. Thus, for any $2 \leq \ell \leq r$ we have $\Delta_{\ell} \leq n^{r-\ell}$. Hence, for any $\tau>0$ we have
\begin{align}\label{Eq:Summation1}
\begin{split}
\Delta(\h^i_j,\tau) &= 2^{{{r}\choose{2}}-1} \sum_{\ell=2}^r \frac{\Delta_{\ell}}{d \tau^{\ell-1}2^{{{\ell-1}\choose{2}}}} \leq n^{0.01kf} \sum_{\ell=2}^r \frac{n^{r-\ell}}{d \tau^{\ell-1}} = n^{0.01kf} \cdot \frac{1}{d \tau^{r-1}} \cdot \sum_{\ell'=0}^{r-2} (n \tau)^{\ell'} \\
&\leq \frac{n^{ks+0.02kf-k}}{\tau^{r-1}} \cdot \sum_{\ell'=0}^{r-2} (n \tau)^{\ell'},
\end{split}
\end{align}
where the term $2^{{{r}\choose{2}}}$ is absorbed into the term $n^{0.01kf}$ using Conditions~(1) and (3).

\medskip \noindent \emph{Choosing $\tau$.} In order to minimize the size of the resulting family of containers, we would like to choose $\tau$ to be as small as possible, subject to the restriction $\Delta(\h^i_j,\tau) \leq \frac{\epsilon}{12r!}$, where we fix $\epsilon = n^{-0.05kf}$ in order to bound the number of steps, as written above. We consider two cases. First we consider the extreme case $s=s_0-0.1f$, in which the value of~\eqref{Eq:Summation1} is the largest and so the restriction on $\tau$ is the strictest, and then we leverage our choice to the general case $s \leq s_0-0.1f$.

\medskip \noindent \emph{The extremal case} $s=s_0-0.1f$. In this case, we choose $\tau = n^{\frac{ks_0-k}{r-1}}$, and so the first term in the summation (i.e., the term that corresponds to $\ell'=0$) becomes $n^{-0.08kf}$. (Clearly, this is the smallest possible value of $\tau$ subject to the restriction, up to the error term $n^{f}$.) By Condition~(1), we have $s_0 \leq \frac{k-r+1}{k}$ and hence $\frac{ks_0-k}{r-1} \leq -1$. Consequently, $n \tau \leq 1$, and thus
\begin{align}\label{Cond:container1}
\begin{split}
\Delta(\h^i_j,\tau) &\leq \frac{n^{k(s_0-0.1f)+0.02kf-k}}{\tau^{r-1}} \cdot \sum_{\ell'=0}^{r-2} (n \tau)^{\ell'} \leq (r-1) \cdot \frac{n^{k(s_0-0.1f)+0.02kf-k}}{\tau^{r-1}} \\
&\leq (r-1) n^{-0.08kf} \leq \frac{n^{-0.05kf}}{12r!} = \frac{\epsilon}{12r!}.
\end{split}
\end{align}
In addition, we have $\tau \leq n^{-1} \leq \frac{1}{200r \cdot (r!)^2}$ by Conditions~(1),(3). Therefore, we can apply Theorem~\ref{thm:container} to the hypergraph $\h^i_j$ with the parameters $(\tau,\epsilon)$ we specified, to obtain a family of containers of size at most
\begin{align}\label{Eq:Upper-number-containers}
\begin{split}
\exp(cN\tau \log(1/\epsilon) \log(1/\tau)) &\leq \exp \left(2000r \cdot (r!)^3 \cdot n^{k-s_0+0.1f} \cdot n^{\frac{ks_0-k}{r-1}} \cdot 0.05kf \log n \cdot \frac{k-ks_0}{r-1} \log n \right) \\
&\leq \exp \left(n^{k-s_0-\frac{k-ks_0}{r-1}+0.2f} \right).
\end{split}
\end{align}

\medskip \noindent \emph{The general case} $s \leq s_0-0.1f$. As the number of steps in our procedure is at most $40/f$, we can choose a sub-optimal value of $\tau$ as long as the size of the family of containers it provides is not larger than the size in the case $s=s_0-0.1f$. Indeed, this increases the total amount of containers by a multiplicative factor of $40/f$ inside the exponent, which can be absorbed into the term $n^{cf}$ inside the exponent using Condition~(2).

Hence, we choose $\tau = n^{\frac{ks_0-k}{r-1}-(s_0-0.1f+s)}$ in order to obtain the same number of containers as in the case $s=s_0-0.1f$. To see that the condition~\eqref{Cond:container1} holds for this choice of $\tau$, we compare $\Delta(\h^i_j,\tau)$ with the corresponding value in the case $s=s_0-0.1f$. Compared to the case $s=s_0-0.1f$, the lower bound on $d$ is decreased by a factor of $n^{k(s_0-0.1f-s)}$, while the term $\tau^{r-1}$ is decreased by a factor of $n^{(r-1)(s_0-0.1f-s)}$. As $r-1<k$ by Condition~(1), it follows that the term $\frac{1}{d \tau^{r-1}}$ (which is the first term in the summation in~\eqref{Eq:Summation1}) is decreased. In addition, as we still have $\tau \leq n^{-1}$, the first term of the summation remains the largest one. (Note that this is where we need Condition~(1). If the condition fails then for $s=s_0-0.1f$ the last term of the summation (i.e., the term which corresponds to $\ell'=r-2$) is the largest one while for small values of $s$ the first term is the largest one. This makes the computations and the final assertion more cumbersome.) Hence, condition~\eqref{Cond:container1} holds in this case as well, and so we can apply Theorem~\ref{thm:container} as in the case $s=s_0-0.1f$ and obtain a family of containers of the same size.

\paragraph{Wrapping up the proof.} We showed that for any $\h^i_j$, the family of containers resulting from applying to it Theorem~\ref{thm:container} with the parameters $(\tau,\epsilon)$ we specified, is of size at most $\exp \left(n^{k-s_0-\frac{k-ks_0}{r-1}+0.2f} \right)$ (see~\eqref{Eq:Upper-number-containers}). As we perform at most $40/f$ steps, the total number of containers in $\C$ is at most
\[
\exp \left( \frac{40}{f} \cdot n^{k-s_0-\frac{k-ks_0}{r-1}+0.2f} \right) \leq \exp \left(n^{k-s_0-\frac{k-ks_0}{r-1}+0.3f} \right).
\]
Since the size of each container is at most $n^{k-s_0+0.1f}$ and any independent set in $\h$ is included in some container in $\C$, the number of independent sets of size $m$ is at most
\[
\exp \left(n^{k-s_0-\frac{k-ks_0}{r-1}+0.3f} \right) \cdot {{n^{k-s_0+0.1f}}\choose{m}},
\]
as asserted.
\end{proof}

\subsection{Proof of Theorem~\ref{thm:main}}
\label{sec:sub:proof}

In this subsection we present the proof of Theorem~\ref{thm:main}. Using Propositions~\ref{Prop:Size_of_graph} and~\ref{Prop:Container}, we show that for an appropriate choice of the parameters $n,k,p$ and $\alpha$, an $\alpha$-random subset $\tilde{S}$ of $[n]^k$ satisfies the conditions of Proposition~\ref{Prop:Reduction}. As explained above, this is sufficient for proving Theorem~\ref{thm:main}.

\begin{proof}[Proof of Theorem~\ref{thm:main}] The proof consists of three steps.

\medskip \noindent \textbf{Step~1: Reformulating the construction of $\tilde{S}$ as an optimization problem.}
Let $q,\eta$ be fixed. Throughout the proof, we assume that $n$ is sufficiently large as function of $q$ (the exact assumption will be specified at the end of the proof; roughly speaking, we shall assume $q \ll \sqrt{\frac{\log n}{\log \log n}}$). We introduce an `error term' $n^{f}$ and use it to absorb all lower-order terms. The value of $f$ will also be chosen at the end of the proof; roughly speaking, it will be of order $\Theta(\eta/q)$.

For some $k,p,\alpha$ to be chosen below, we consider an $\alpha$-random subset $\tilde{S}$ of $[n]^k$ (i.e., a subset of $[n]^k$ in which each point is chosen with probability $\alpha$, independently of other points).

\medskip \noindent \emph{The conditions the random subset has to satisfy.} The parameters have to be chosen such that the following conditions hold with a high probability:

\smallskip \noindent (1). $\tilde{S}$ does not contain any independent set of $\h(n,k,q)$ of size $p$.

\smallskip \noindent (2). $\tilde{S}$ is of size at least $\alpha n^k /2$ and contains at most $\alpha n^k/4$ collinear $u$-tuples.

\medskip \noindent If both conditions hold with a high probability, then we can find an explicit set $\tilde{S}$ of points that satisfies both of them, remove from $\tilde{S}$ one point from each collinear $u$-tuple, and obtain a set $S$ of size at least $\alpha n^k/4$ with no collinear $u$-tuples. This is the set $S$ required in Proposition~\ref{Prop:Reduction}.

\medskip \noindent \emph{The function we want to optimize.} Recall that by Proposition~\ref{Prop:Reduction}, the set $S$ we obtain by the probabilistic process can be transformed into a family $\F$ that satisfies the $(p,q)$ property and cannot be pierced by less than $\frac{\alpha n^k}{4(u-1)}$ points.  Hence, in order to obtain the strongest lower bound we can for the $(p,q)$ theorem, we want to make $\frac{\alpha n^k}{4(u-1)}$ as large as possible with respect to $p$, subject to the above conditions. Specifically, letting $T=T(p,q,n,\alpha,u)$ be such that $\frac{\alpha n^k}{4(u-1)}=p^T$, we are interested in maximizing
\begin{equation}\label{Eq:Opt}
T = \log_p \left(\frac{\alpha n^k}{4(u-1)} \right),
\end{equation}
and we want to show that the parameters can be chosen such that
\begin{equation}\label{Eq:T}
T \geq 1+\frac{1-\eta}{4q-7}.
\end{equation}

\medskip \noindent \emph{Modifying the conditions using Propositions~\ref{Prop:Size_of_graph} and~\ref{Prop:Container}.}

\medskip \noindent \textbf{Condition~(1).} In order to achieve this condition, it is clearly sufficient that the expected number of independent sets of size $p$ in an $\alpha$-random subset of $\h(n,k,q)$ is $o(1)$. By Proposition~\ref{Prop:Container} (applied with $f$ and a parameter $s_0$ to be determined below), the number of independent sets of size $p$ in $\h(n,k,q)$ is at most $\exp \left(n^{k-s_0-\frac{k-ks_0}{q-1}+0.3f} \right) \cdot {{n^{k-s_0+0.1f}}\choose{p}}$. For each such set, the probability that it is included in $\tilde{S}$ is $\alpha^p$. Thus, by linearity of expectation, a sufficient condition is
\[
\exp \left(n^{k-s_0-\frac{k-ks_0}{q-1}+0.3f} \right) \cdot {{n^{k-s_0+0.1f}}\choose{p}} \cdot \alpha^p = o(1).
\]
Using the standard inequalities ${{N}\choose{\ell}} \leq \frac{N^{\ell}}{\ell!}$ and $\ell! \geq (\frac{\ell}{e})^{\ell}$, this implies that a sufficient condition is
\begin{equation}\label{Eq:Cond-p-independent}
\exp \left(n^{k-s_0-\frac{k-ks_0}{q-1}+0.3f} \right) \cdot \left( \frac{e \cdot n^{k-s_0+0.1f}}{p} \right)^p \cdot \alpha^p = o(1).
\end{equation}
Note that this modification is valid only if $n$ and the `error term' $f$ are chosen in such a way that Proposition~\ref{Prop:Container} can be applied; we shall see that this restriction is the main source of the hypothesis on the relation between $p$ and $q$ in the formulation of the theorem.

\medskip \noindent \textbf{Condition~(2).} The condition $\Pr[|\tilde{S}| \geq \alpha n^k /2]=1-o(1)$ holds by a standard tail bound for Binomial random variables, unless $\alpha$ is extremely small. (We will verify it formally at the end of the proof for the specific value of $\alpha$ we choose.) By Markov's inequality, in order to prove that with a high probability, $\tilde{S}$ contains at most $\alpha n^k/4$ collinear $u$-tuples, it is sufficient to show that the expected number of collinear $u$-tuples in an $\alpha$-random subset of $[n]^k$ is $o(\alpha n^k)$. By Proposition~\ref{Prop:Size_of_graph} (assuming $u \geq k+1$; it is easy to check that choosing $u \leq k$ leads to worse results), the number of collinear $u$-tuples in $[n]^k$ (which is exactly the number of hyperedges of the hypergraph $\h(n,k,u)$) is at most $\frac{k \cdot 2^{u+k}}{u!} \cdot n^{u+k-1} \log n$. For each such $u$-tuple, the probability that it is included in $\tilde{S}$ is $\alpha^u$. Hence, by linearity of expectation, a sufficient condition is
\begin{equation}\label{Eq:Cond-u-collinear}
\frac{k \cdot 2^{u+k}}{u!} \cdot n^{u+k-1} \log n \cdot \alpha^u = o(\alpha n^k).
\end{equation}


\medskip \noindent \textbf{Step~2: Choosing the parameters.} In this step we choose the parameters one-by-one aiming at optimizing~\eqref{Eq:Opt}, subject to the single restriction~\eqref{Eq:Cond-p-independent}. We then show that the choice of parameters we obtain satisfies the second restriction~\eqref{Eq:Cond-u-collinear}, which is of course sufficient.

For sake of clarity, we \emph{omit the error term $n^{f}$ during this step} and introduce it back once all parameters are set (which is sufficient for verifying formally that our construction of $\tilde{S}$ indeed satisfies the hypothesis of Proposition~\ref{Prop:Reduction}). In addition, we note that since the parameter $u$ we choose satisfies $u \leq \log n$ (as we show below), the term $\log_n(4(u-1))$ in the target function~\eqref{Eq:Opt} can be absorbed into the error term (this holds unless the error term is extremely small; we will verify this formally below after the error term will be specified). Hence, we omit it and simplify the target function to
\begin{equation}\label{Eq:Opt2}
T = \log_p (\alpha n^k) = \frac{k+\log_n \alpha}{\log_n p}.
\end{equation}

\medskip \noindent \emph{The choice of $p$ and $\alpha$.} Assume that the parameters $k$ and $s_0$ are fixed and we want to choose $p$ and $\alpha$ optimally. Note that in order to satisfy~\eqref{Eq:Cond-p-independent}, we should choose $p,\alpha$ in such a way that the term $\alpha^p$ cancels the two former terms of~\eqref{Eq:Cond-p-independent}.

Let us choose some value of $p$ and set $\beta$ such that $p=n^{k-s_0-\beta}$. On the one hand, in order to use $\alpha^p$ to cancel the first term in~\eqref{Eq:Cond-p-independent}, we must have $k-s_0-\beta > k-s_0-\frac{k-ks_0}{q-1}$ (up to addition of lower-order terms), or equivalently,
\begin{equation}\label{Eq:beta-1}
\beta< \frac{k-ks_0}{q-1}.
\end{equation}
On the other hand, in order to cancel the second term in~\eqref{Eq:Cond-p-independent}, we must have $\alpha<n^{-\beta}$. In such a case, our target function satisfies
\[
T < \frac{k-\beta}{k-s_0-\beta},
\]
and approaches $\frac{k-\beta}{k-s_0-\beta}$ as $\alpha$ increases to $n^{-\beta}$. Assuming we take $\alpha = n^{-\beta}$ (up to lower-order terms) in order to maximize $T$, we would like to choose $\beta$ such that the expression $\frac{k-\beta}{k-s_0-\beta}$ is maximized.

It is easy to check that the function $\beta \mapsto \frac{k-\beta}{k-s_0-\beta}$ is monotone increasing, and thus, in order to maximize $T$ we would like to choose $\beta$ to be as large as possible.

However, by~\eqref{Eq:beta-1} we must take $\beta \leq \frac{k-ks_0}{q-1}$. Thus, we choose
\begin{equation}\label{Eq:p-1}
\alpha = n^{-\frac{k-ks_0}{q-1}} \qquad \mbox{ and } \qquad p=n^{k-s_0-\frac{k-ks_0}{q-1}}
\end{equation}
(up to the error term $n^{f}$).

\medskip \noindent \emph{The choice of $s_0$.} Now we assume that only the parameter $k$ is fixed and we want to choose $s_0$ optimally. Recall that following Proposition~\ref{Prop:Container}, we assume $s_0 \leq \frac{k-q+1}{k}$; the case $s_0 > \frac{k-q+1}{k}$ is considered in Appendix~\ref{app:Large-s_0}.

Choosing $\alpha$ and $p$ as in~\eqref{Eq:p-1}, the target function $T$ becomes
\[
T = \frac{k- \frac{k-ks_0}{q-1}}{k-s_0-\frac{k-ks_0}{q-1}} = 1+ \frac{s_0}{k-s_0-\frac{k-ks_0}{q-1}} = 1+ \left(\frac{q-1}{k-q+1}\right) \cdot \left(\frac{s_0}{c+s_0} \right),
\]
where $c = (k-\frac{k}{q-1}) \cdot \frac{q-1}{k-q+1}$ is fixed. (Note that $k-q+1 > 0$ by the assumption on $s_0$). It is easy to check that the function $s_0 \mapsto \frac{s_0}{c+s_0}$ is monotone increasing, and hence in order to maximize the target function we have to take $s_0$ as large as possible.
As by assumption, $s_0 \leq \frac{k-q+1}{k}$, we choose
\begin{equation}\label{Eq:s_0-1}
s_0 = \frac{k-q+1}{k}.
\end{equation}
By~\eqref{Eq:p-1}, this implies
\begin{equation}\label{Eq:p-2}
\alpha = n^{-1} \qquad \mbox{ and } \qquad p=n^{k-1-\frac{k-q+1}{k}}
\end{equation}
(up to the error term $n^{f}$).


\medskip \noindent \emph{The choice of $k$.} Now we want to choose $k$ optimally. Choosing $s_0,\alpha,$ and $p$ as in~\eqref{Eq:p-1},~\eqref{Eq:p-2}, the target function $T$ becomes
\begin{equation}\label{Eq:Target-last}
T = T(k,q) = 1+ \frac{\frac{k-q+1}{k}}{k-\frac{k-q+1}{k}-1}.
\end{equation}
It is easy to check that for a fixed $q$, the function $k \mapsto T(k,q)$ attains a maximum at $k=k_0 = (q-1) + \sqrt{(q-1)(q-2)}$, is increasing for $(q-1)-\sqrt{(q-1)(q-2)} < k < k_0$ and is decreasing for $k>k_0$. Hence, $\max_{k \in \mathbb{N}} T(k,q)$ is attained either for $k=2q-3$ or for $k=2q-2$. Substituting into~\eqref{Eq:Target-last}, we see that $T(2q-2,q)=T(2q-3,q)=1+\frac{1}{4q-7}$, which is indeed the lower bound on $T$ we wanted to obtain (up to an additive error term of $O(\frac{\eta}{q})$; see~\eqref{Eq:T}). Hence, we choose
\begin{equation}\label{Eq:k-1}
k=2q-2.
\end{equation}
By~\eqref{Eq:p-2}, this implies
\begin{equation}\label{Eq:p-3}
s_0=0.5, \qquad \alpha = n^{-1}, \qquad \mbox{ and } \qquad p=n^{2q-3.5}
\end{equation}
(up to the error term $n^{f}$).

\medskip \noindent \textbf{Remark.} Note that we can choose $k=2q-3$, which is the natural generalization of the choice $k=3$ made in the case $q=3$ in~\cite{BS17}. This would lead to the same results, but the calculations become more cumbersome. Specifically, instead of $s_0=0.5$ we would obtain $s_0=\frac{q-2}{2q-3}$, and other terms would look more complex as well. Hence, we prefer choosing $k=2q-2$.

\medskip \noindent \textbf{Step~3: Wrapping up the proof.} Following Step~2 and re-introducing the error term $n^f$, we choose the parameters $k=2q-2$, $s_0 = 0.5$, $\alpha=n^{-1-f}$, and $p=n^{2q-3.5+f}$. That is, we consider an $(n^{-1-f})$-random subset $\tilde{S}$ of $[n]^{2q-2}$ and claim that with a high probability it satisfies Conditions~(1) and~(2). Note that we do not choose the values of $f,n$ yet; we delay this choice to the end of the proof, where the intuition behind it will become apparent. However, we stress that $f,n$ will be chosen in such a way that the hypotheses of Proposition~\ref{Prop:Container} will be satisfied, and we use these hypotheses in the calculations below.

\medskip \noindent \emph{Verifying the conditions.} As shown above, in order to prove that Condition~(1) is satisfied, it is sufficient to show that~\eqref{Eq:Cond-p-independent} holds. Indeed, we have
\begin{align*}
\exp \left(n^{k-s_0-\frac{k-ks_0}{q-1}+0.3f} \right) &\cdot \left( \frac{e \cdot n^{k-s_0+0.1f}}{p} \right)^p \cdot \alpha^p \\
&= \exp \left(n^{2q-3.5+0.3f} \right) \cdot (en^{1-0.9f})^{n^{2q-3.5+f}} \cdot n^{(-1-f)n^{2q-3.5+f}}=o(1).
\end{align*}
In order to prove that Condition~(2) holds, it is sufficient to show that $\Pr[|\tilde{S}| \geq \alpha n^k /2]=1-o(1)$ and that~\eqref{Eq:Cond-u-collinear} holds. The former holds for $\alpha=n^{-1-f}$ by a standard tail estimate for Binomial random variables. As for the latter, we have
\[
\frac{k \cdot 2^{u+k}}{u!} \cdot n^{u+k-1} \log n \cdot \alpha^u = \frac{(2q-2)2^{2q-2+u}}{u!} \cdot n^{u+2q-3} \log n \cdot n^{(-1-f)u} \leq n^{2q-3+(f-fu)} = o(\alpha n^{k})
\]
for any $u \geq 3$, since $\frac{(2q-2)2^{2q-2+u}}{u!} \log n \leq n^f$ for a sufficiently large $n$ as function of $f$. (Note that this is actually another condition of $f$; however, this condition must be satisfied if $f$ is chosen in such a way that Proposition~\ref{Prop:Container} can be applied.) As for applying Proposition~\ref{Prop:Size_of_graph} we assumed $u \geq k+1$, we may choose $u=2q-1$, and so~\eqref{Eq:Cond-u-collinear} indeed holds.

\medskip \noindent \emph{Choosing the error term $f$ and deducing the required relation between $p$ and $q$.} For our choice of parameters we have
\[
\log_p \left(\frac{\alpha n^k}{4(u-1)} \right) = \frac{\log_n \left(\frac{\alpha n^k}{4(u-1)} \right)}{\log_n p} \geq \frac{2q-3-1.1f}{2q-3.5+f} \geq 1+\frac{1}{4q-7}-f,
\]
where the last inequality holds for any $q \geq 3$ and $f \leq 1$. Hence, in order to obtain the asserted bound $T \geq 1+\frac{1-\eta}{4q-7}$ we have to choose $f$ such that $f<\frac{\eta}{4q-7}$. Hence, we choose
\begin{equation}\label{Eq:f}
f=\frac{\eta}{4q}.
\end{equation}
This allows us to compute the required restriction on the relation between $n$ and $q$. Recall that we can use Proposition~\ref{Prop:Container} if three conditions are satisfied. The third condition is $k \leq 0.001f \cdot \frac{\log n}{\log \log n}$. As $f= \frac{\eta}{4q}$, this condition becomes
\[
(2q-2)4q \leq \frac{0.001 \eta \log n}{\log \log n}.
\]
To satisfy this condition we make the assumption $q \leq 0.01 \eta \sqrt{\frac{\log n}{\log \log n}}$. Finally, since $p=n^{2q-3.5+f}<n^{2q}$, a sufficient requirement on the relation between $p$ and $q$ is $q \leq 0.01 \eta (\frac{\log p}{\log \log p})^{1/3}$, and this is indeed the assumption in Theorem~\ref{thm:main}.

\medskip \noindent \emph{Concluding the proof.} We conclude that for any $(p,q)$ such that $q \leq 0.01\eta (\frac{\log p}{\log \log p})^{1/3}$, and for the values of the other parameters described above, the choice of $\tilde{S}$ satisfies Conditions~(1) and~(2) and the target function $T$ satisfies $T \geq 1+\frac{1-\eta}{4q-7}$. Therefore, the set $S$ (obtained from a specific choice of $\tilde{S}$ by removing one point from each collinear $u$-tuple) indeed satisfies the hypothesis of Proposition~\ref{Prop:Reduction}. This completes the proof of Theorem~\ref{thm:main}.
\end{proof}

\subsection{An application to a hypergraph coloring problem}
\label{sec:sub:coloring}

In this subsection we present the proof of Proposition~\ref{Prop:Coloring}. Let us recall the statement of the problem.

For a set $P$ of $m$ points in the plane, we let $H_{q}(P)$ be the (non-uniform) hypergraph whose vertex set is $P$ and whose hyperedges are all sets
\[
\{S \subset P: S = P \cap \ell \mbox{ for some line $\ell$ }, |S| \geq q\}.
\]
Our goal is to find the maximal possible chromatic number of $H_q(P)$ as function of $m$, i.e., determine
\[
g_q(m) = \max_{|P|=m} \chi(H_q(P)).
\]
We observe that the construction of Balogh and Solymosi~\cite{BS17} can be used to obtain the lower bound $g_3(m) > m^{\frac{1}{6}-f}$ for any $f>0$ and any $m>m_0(f)$. Indeed, let $f>0$ and let $P$ be a set of $m$ points in the plane such that $S$ does not contain 4 collinear points, and any subset of $S$ of size $m^{\frac{5}{6}+f}$ contains a collinear triple (as constructed in~\cite{BS17}). Consider a proper coloring of the hypergraph $H_3(p)$ with $c$ colors. If some color set contains a collinear triple, then it contains in full some hyperedge of $H_3(p)$ (since $P$ does not contain collinear $q$-tuples for $q>3$). Hence, each color set is of size $<m^{\frac{5}{6}+f}$ , and therefore, $c>m^{\frac{1}{6}-f}$.

In order to generalize this lower bound to a bound for arbitrary $q \geq 3$, we have to construct a set $S \subset \mathbb{R}^2$ such that $S$ does not contain $q+1$ collinear points, while any subset of $S$ of size $t(m,q)$ contains a collinear $q$-tuple, where we want $t(m,q)$ to be as small as possible. Note that for the above argument to work, it is crucial that $S$ contains no collinear $(q+1)$-tuples. Therefore, we cannot use our generalized construction from the proof of Theorem~\ref{thm:main}, as in that construction, the set $S$ is only guaranteed to be free of collinear $(2q-2)$-tuples. In other words, in the choice of parameters we are forced to choose $u=q+1$.

Fortunately, the analysis of the choice of parameters presented above allows us to choose the parameters under this additional restriction. We omit the details here and only note that the step which should be modified is the choice of $k$, and the target function that should be optimized becomes
\[
k \mapsto \frac{k-\frac{k}{q}}{k-\frac{k-q+1}{k}-1},
\]
in the range $k \geq q$. It can be easily checked that in the examined range, the function is decreasing, and so we choose $k=q$. As a result, the target function becomes $T=1+\frac{1}{q^2-q-1}$, and the error term $f$ is chosen to be $\frac{\eta}{q^2}$, which leads to the restriction $q \leq 0.005\eta (\frac{\log m}{\log \log m})^{1/4}$. Formally, we obtain the following corollary of our proof method:
\begin{proposition}\label{Prop:Equal}
For any $0<\eta<1/2$ and for any $p,q \geq 3$ such that $q \leq 0.01\eta (\frac{\log p}{\log \log p})^{1/4}$, there exists a family $S \subset \mathbb{R}^2$ of size $p^{1+\frac{1-\eta}{q^2-q-1}}$ such that $S$ does not contain $q+1$ collinear points, while any subset of $S$ of size $p$ contains a collinear $q$-tuple.
\end{proposition}

Now we are ready to present the proof of Proposition~\ref{Prop:Coloring}. Let us recall its formulation.

\medskip \noindent \textbf{Proposition~\ref{Prop:Coloring}.} For any $\eta>0$, $q \geq 3$ and any $m \in \mathbb{N}$ such that $q \leq 0.005\eta (\frac{\log m}{\log \log m})^{1/4}$, we have
\[
g_q(m) \geq m^{\frac{1-\eta}{q^2-q-\eta}}.
\]

\begin{proof}[Proof of Proposition~\ref{Prop:Coloring}]
Let $p=m^{1-\frac{1-\eta}{q^2-q-\eta}}$. It is easy to check that $q \leq 0.01\eta (\frac{\log p}{\log \log p})^{1/4}$, and thus, Proposition~\ref{Prop:Equal} can be applied with $\eta$ and the pair $(p,q)$. The proposition asserts the existence of a set $P$ of size $p^{1+\frac{1-\eta}{q^2-q-1}}=m$, such that any subset of $P$ of size $p$ contains a collinear $q$-tuple. As $P$ (provided by the proposition) does not contain collinear $(q+1)$-tuples, this implies that in any proper coloring of the hypergraph $H_q(P)$, the size of each color class is less than $p=m^{1-\frac{1-\eta}{q^2-q-\eta}}$. Hence, $\chi(H_q(P)) > m^{\frac{1-\eta}{q^2-q-\eta}}$. This completes the proof.
\end{proof}

\subparagraph*{Acknowledgements}

The authors are grateful to Charles Wolf for valuable discussions and suggestions.

\bibliographystyle{alpha}
\bibliography{references}

\appendix

\section{Analysis of the case $s_0 > \frac{k-q+1}{k}$}
\label{app:Large-s_0}

The choice of parameters in the proof of Theorem~\ref{thm:main} made an extra assumption: $s_0 \leq \frac{k-q+1}{k}$. In this appendix we complement the proof by showing that choosing $s_0>\frac{k-q+1}{k}$ does not lead to a better lower bound in the $(p,q)$ theorem.

Essentially, the argument goes as follows. We first show that if $s_0$ is increased then while the maximal size of containers decreases, the total number of containers $|\C|$ increases. Then we show that the increase of $|\C|$ forces us to increase $p$ (thus, obtaining a $(p,q)$ property with a larger $p$). In order to compensate for the increase of $p$, we have to increase $\alpha$ (thus, increasing the size of the set $S$ which satisfies the $(p,q)$ property). We show that this increase necessitates us to increase $u$, in such a way that the lower bound $|S|/(u-1)$ on the piercing number of $S$ does not increase, and so we obtain a worse relation between the piercing number $\HD_2(p,q)$ and $p$. The argument is given in full in the following proposition.

\begin{proposition}[Informal]
Taking $s_0 > \frac{k-q+1}{k}$ does not lead to an improved lower bound for the $(p,q)$ theorem.
\end{proposition}

\begin{proof}[Informal proof.]

Recall that the target function we want to optimize is
\begin{equation}\label{Eq:Opt-App}
T = \log_p \left(\frac{\alpha n^k}{4(u-1)} \right).
\end{equation}

\paragraph{Step 1: The value of $p$ is not decreased.} We claim that the value of $p$ cannot be decreased by increasing $s_0$. Indeed, increasing $s_0$ means that we apply the hypergraph container theorem more times than in the case $s_0 \leq \frac{k-q+1}{k}$ (specifically, until all container sets become smaller than $n^{k-s_0+0.1f} < n^{k-\frac{k-q+1}{k}+0.1f}$). This can be viewed as applying Theorem~\ref{thm:container} until all containers become as small as $n^{k-\frac{k-q+1}{k}+0.1f}$ and then applying it several more times. Thus, the total number of containers $\C$ does not decrease.

Recall that Condition~(1) which the random choice of $\tilde{S}$ must satisfy is that with a high probability, $\tilde{S}$ does not contain an independent set of size $p$ of $\h(n,k,q)$. By applying the hypergraph container theorem, we replaced this condition by
\begin{equation}\label{Eq:Large-s_0-0}
|\C| \cdot {{n^{k-s_0+0.1f}}\choose{p}} \cdot \alpha^p = o(1)
\end{equation}
(see~\eqref{Eq:Cond-p-independent}; note that the value $|\C|$ replaces the value $\exp \left(n^{k-s_0-\frac{k-ks_0}{q-1}+0.3f} \right)$ in~\eqref{Eq:Cond-p-independent}, since in our case we cannot apply Proposition~\ref{Prop:Container}, as we now assume $s_0>\frac{k-q+1}{k}$).

As shown in the argument explaining the choice of $p$ in Step~2 of the proof of Theorem~\ref{thm:main} (which does apply for any choice of $s_0$), in order to satisfy~\eqref{Eq:Large-s_0-0}, we must choose $p,\alpha$ in such a way that the term $\alpha^p$ cancels the two former terms of~\eqref{Eq:Large-s_0-0}. In particular, in order to use $\alpha^p$ to cancel the first term in~\eqref{Eq:Large-s_0-0}, we must have
$\log_n p > \log_n (|\C|)$. In the proof of Theorem~\ref{thm:main}, we eventually choose $p$ in such a way that $\log_n p = \log_n (|\C|)$, up to the error term $n^{f}$.

Since in our case, $|\C|$ is not decreased, this implies that $p$ cannot be decreased (up to the error term $n^{f}$).

\paragraph{Step 2: In order to increase $\alpha$ we must increase $u$ accordingly.} As $p$ is not decreased, in order to increase the target value $T$ we must increase the value $\alpha$. (Note that decreasing $u$ essentially does not increase $T$ as in the proof of Theorem~\ref{thm:main} $u$ is chosen to be $2q-1<\log n$, and thus its influence on $T$ is negligible.) However, increasing $\alpha$ makes it more complex for our choice of $\tilde{S}$ to satisfy Condition~(2) (which asserts that with a high probability, $\tilde{S}$ contains only a `small' amount of collinear $u$-tuples). Specifically, as shown in the first step of the proof of Theorem~\ref{thm:main}, in order to satisfy Condition~(2) we need
\[
\frac{k \cdot 2^{u+k}}{u!} \cdot n^{u+k-1} \log n \cdot \alpha^u = o(\alpha n^k).
\]
Equivalently, we need
\begin{equation}\label{Eq:Large-s_0-1}
\frac{k \cdot 2^{u+k}}{u!} \cdot \log n = o \left( (\alpha n)^{1-u} \right).
\end{equation}
Using the standard inequality $u! < eu \cdot (u/e)^u$ and absorbing the terms $e \cdot k \cdot 2^k$ and $\log n$ into the error term $n^{f}$ which is not displayed here,~\eqref{Eq:Large-s_0-1} simplifies to
\[
\frac{(2e)^u}{u^{u+1}} = o \left( (\alpha n)^{1-u} \right),
\]
or equivalently,
\[
\alpha n = o \left( \frac{u^{(u+1)/(u-1)}}{(2e)^{u/(u-1)}} \right) = o(u).
\]
Therefore, we must have
\begin{equation}\label{Eq:Large-s_0-2}
\frac{\alpha}{4(u-1)} = o(n^{-1}).
\end{equation}

\paragraph{Step 3: Overall, we do not obtain a stronger $(p,q)$ theorem.} Substituting~\eqref{Eq:Large-s_0-2} into the target function $T$ (see~\eqref{Eq:Opt-App}), we obtain
\[
T = \log_p \left(\frac{\alpha n^k}{4(u-1)} \right) \leq \log_p n^{k-1}.
\]
That is, while $p$ increases, the term inside the logarithm is not larger than $n^{k-1}$, which is (up to the error term $n^{f}$) the term inside the logarithm chosen in the proof of Theorem~\ref{thm:main}. Therefore, the target function $T$ does not increase, compared to its value in Theorem~\ref{thm:main}, as asserted.
\end{proof}

\end{document}